\documentclass[copyright]{eptcs}
\usepackage{breakurl} 
\usepackage{underscore}
\usepackage{preamble}

\newcommand{\C}{\smallCat{C}}

\title{Coequalisers under the Lens}
\def\titlerunning{Coequalisers under the Lens}

\author{Matthew Di Meglio
\institute{Department of Mathematics and Statistics\\
Macquarie University\\
Sydney, Australia}
}
\def\authorrunning{Matthew Di Meglio}

\providecommand{\event}{ACT 2021}

\begin{document}
\maketitle

\begin{abstract}
Lenses encode protocols for synchronising systems. We continue the work begun by Chollet et al.\ at the Applied Category Theory Adjoint School in 2020 to study the properties of the category of small categories and asymmetric delta lenses. The forgetful functor from the category of lenses to the category of functors is already known to reflect monos and epis and preserve epis; we show that it preserves monos, and give a simpler proof that it preserves epis. Together this gives a complete characterisation of the monic and epic lenses in terms of elementary properties of their get functors.

Next, we initiate the study of coequalisers of lenses. We observe that not all parallel pairs of lenses have coequalisers, and that the forgetful functor from the category of lenses to the category of functors neither preserves nor reflects all coequalisers. However, some coequalisers are reflected; we study when this occurs, and then use what we learned to show that every epic lens is regular, and that discrete opfibrations have pushouts along monic lenses. Corollaries include that every monic lens is effective, every monic epic lens is an isomorphism, and the class of all epic lenses and the class of all monic lenses form an orthogonal factorisation system.
\end{abstract}

\section{Introduction}

A \textit{bidirectional transformation} between two systems is a specification of when the joint state of the two systems should be regarded as consistent, together with a protocol for updating each system to restore consistency in response to a change in the other~\cite{gibbons:2018:bidirectionaltransformations}. The study of bidirectional transformations goes back to as far as 1981 with Bancilhon and Spyrato’s work on the view-update problem for databases~\cite{bancilhon:1981:updatesemanticsrelationalviews}. The view-update problem is about \textit{asymmetric} bidirectional transformations; those where the state of one of the systems, called the \textit{view}, is completely determined by that of the other, called the \textit{source}. Bidirectional transformations also arise in many other contexts across computer science, such as when programming with complex data structures and when linking user interfaces to data models.

An \textit{asymmetric state-based lens} is a mathematical encoding of an asymmetric bidirectional transformation in which the consistency restoration updates to the source are assumed to be dependent only on the old source state and the updated view state. If \(S\) is the set of source states and~\(V\) is the set of view states, such a lens consists of a \textit{get function} \(S \to V\) and a \textit{put function} \(S \times V \to S\) which, ideally, satisfy certain laws. The earliest known account of asymmetric state-based lenses may be found in Oles’ PhD thesis~\cite[Chapter~VI]{oles:1982:CategoryApproachSemanticsProgrammingLanguages}, where they are called \textit{extensions} of \textit{store shapes}; they are a key ingredient in Oles' semantics for an imperative stack-based programming language with block-scoped variables because they capture the essential properties of a data store which changes shape as variables come into and go out of scope. All recent notions of lens, including the name \textit{lens}, may be traced back to the work of Pierce~et~al.~\cite{pierce:2007:combinatorsbidirectionaltreetransformations}; they proposed variants of asymmetric state-based lenses for modelling bidirectional transformations on tree-structured data, and they also introduced the idea of building lenses compositionally with a domain-specific language such as their lens combinators.

Diskin~et~al.\ highlighted the inadequacy of state-based lenses as a general mathematical model for bidirectional transformations~\cite{diskin:2011:statetodeltabx}, providing several examples of situations in which consistency restoration would benefit from knowing more about each change to the view than just the view's new state. In an \textit{asymmetric delta lens}, their proposed alternative, systems are modelled as categories of states and transitions (deltas) rather than simply as sets of states. Also, the put operation takes as input specifically which transition occurred in the view rather than just the end state of that transition.

Application of category theory to the study of lenses has already proved fruitful. Johnson and Rosebrugh's research program~\cite{JohnsonRosebrugh:2015:SpansDeltaLenses, JohnsonRosebrugh:2016:UnifyingSetBasedDeltaBasedEditBasedLenses, JohnsonRosebrugh:2017:UniversalUpdatesForSymmetricLenses} has enabled a unified treatment of symmetric and asymmetric delta lenses, with the perspective that a symmetric delta lens is an equivalence class of spans of asymmetric delta lenses. Ahman and Uustalu's observation that asymmetric delta lenses are compatible functor cofunctor pairs~\cite{AhmanUustalu:2017:TakingUpdatesSeriously}, and Clarke's generalisation of these lenses to the internal category theory setting~\cite{Clarke:2020:InternalLensesAsFunctorsAndCofunctors}, have enabled an abstract diagrammatic approach to proofs involving these lenses~\cite{Clarke:2021:ADiagrammaticApproachToSymmetricLenses}, in which we may profit from the already well-developed theory of functors and opfibrations. Yet, until the work of Chollet et al.~\cite{Clarke:2021:CategoryLens}, little was known about the category of asymmetric delta lenses. Building on their work, this paper aims to further our understanding of this category.

\subsection*{Outline}

Henceforth, we refer to asymmetric delta lenses simply as \textit{lenses}, which we formally define in \cref{Section: Background}.

In \cref{Section: Characterising monic and epic lenses}, we prove the conjecture by Chollet et al.~\cite{Clarke:2021:CategoryLens} that the forgetful functor from the category of lenses to the category of functors preserves monos. Together with their result that it reflects monos, we deduce that the monic lenses are the unique lenses on cosieves; these are equivalently the out-degree-zero subcategory inclusion functors. We also provide a proof, simpler than the original one sketched by Lack in an unpublished personal communication to Clarke, that the forgetful functor preserves epis.

In \cref{Section: Coequalisers of lenses}, we initiate the study of coequalisers of lenses. We begin with examples of how they are not as well behaved as one might hope; specifically, not all parallel pairs of lenses have coequalisers, and the forgetful functor neither preserves nor reflects all coequalisers. We then prove our main result, \cref{Condition for reflection}, which is about the coequalisers that are actually reflected by the forgetful functor.

In \cref{Section: Pushouts of discrete opfibrations along monos}, we use \cref{Condition for reflection} to show that the category of lenses has pushouts of discrete opfibrations along monos. We then show that every monic lens is effective. It follows that the classes of all monos, all effective monos, all regular monos, all strong monos and all extremal monos in the category of lenses coincide, and thus also that all lenses which are both monic and epic are isomorphisms.

In \cref{Section: Regular epic lenses}, we use \cref{Condition for reflection} again to show that every epic lens is regular. It follows that the classes of all epis, all regular epis, all strong epis and all extremal epis in the category of lenses coincide. It also follows that the class of all epic lenses is left orthogonal to the class of all monic lenses. Together with other known results, this means that they form an orthogonal factorisation system.

\section{Background}
\label{Section: Background}

\subsection*{Notation}
\label{Section: Notation}

Application of functions (functors, lenses, etc.)\ is written by juxtaposing the function name with its argument. Application is right associative, so an expression like \(FGx\) parses as \(F (Gx)\) and not \((FG)x\). Parentheses are only used when needed or for clarity. Binary operators like \(\compose\) have lower precedence than application, so an expression like \(F a \compose F b\) parses as \((F a) \compose (F b)\) and not \(F \paren[\big]{(a \compose F) b}\).

Let \(\Cat\) denote the category whose objects are small categories and whose morphisms are functors. Categories with boldface names \(\A\), \(\B\), \(\C\), etc.\ are always small. We write \(\objectSet{\C}\) for the set of objects of a small category \(\C\), and, for all \(X, Y \in \objectSet{\C}\), we write \(\homSet{\C}{X}{Y}\) for the set of morphisms of \(\C\) from \(X\) to \(Y\). For each \(X \in \objectSet{\C}\), we write \(\outSet{\C}{X}\) for the set \(\DisjointUnion_{Y \in \objectSet{\C}}\homSet{\C}{X}{Y}\) of all morphisms in \(\C\) out of \(X\). We write \(\source{f}\) and \(\target{f}\) for, respectively, the source and target of a morphism \(f\). We also write \(f \colon X \to Y\) to say that \(X, Y \in \objectSet{\C}\) and \(f \in \homSet{\C}{X}{Y}\). The composite of morphisms \(f \colon X \to Y\) and \(g \colon Y \to Z\) is denoted \(g \compose f\).

The category with a single object \(0\) and no non-identity morphisms, also known as the \textit{terminal category}, is denoted \(\terminalCat\). The category with two objects \(0\) and \(1\) and a single non-identity morphism, namely \(u \colon 0 \to 1\), also known as the \textit{interval category}, is denoted \(\intervalCat\). The category with two objects \(0\) and \(1\) and two non-identity morphisms, namely \(v \colon 0 \to 1\) and \(v^{-1} \colon 1 \to 0\), also known as the \textit{free living isomorphism}, is denoted \(\isomorphismCat\). We will identify objects and morphisms of a small category \(\C\) with the corresponding functors \(\terminalCat \to \C\) and \(\intervalCat \to \C\) respectively.

If the square
\begin{equation}
    \label{Equation: Pullback square}
    \begin{tikzcd}
        \D \arrow[r, "T"]\arrow[d, "S" swap] & \B \arrow[d, "G"]\\
        \A \arrow[r, "F" swap] & \C
    \end{tikzcd}
\end{equation}
in \(\Cat\) is a pushout square and \(F' \colon \A \to \E\) and \(G' \colon \B \to \E\) are functors for which \(F' \compose S = G' \compose T\), then we write \(\copair{F'}{G'}\) for the functor \(\C \to \E\) induced from \(F'\) and \(G'\) by the universal property of the pushout. Similarly, if the square \eqref{Equation: Pullback square} in \(\Cat\) is a pullback square and \(S' \colon \E \to \A\) and \(T' \colon \E \to \B\) are functors for which \(F \compose S' = G \compose T'\), then we write \(\pair{S'}{T'}\) for the functor \(\E \to \D\) induced from \(S'\) and \(T'\) by the universal property of the pullback. By our identification of objects with functors from~\(\terminalCat\) mentioned above, if \(A \in \objectSet{\A}\) and \(B \in \objectSet{\B}\) are such that \(FA = GB\), then \(\pair{A}{B}\) is the object of \(\D\) selected by the functor \(\terminalCat \to \D\) induced by the universal property of the pullback from the functors \(\terminalCat \to \A\) and \(\terminalCat \to \B\) that respectively select the objects \(A\) and \(B\).

\subsection*{Lenses and discrete opfibrations}
\label{Section: Lenses and discrete opfibrations}

First, we recall the definition of a (asymmetric delta) lens~\cite{diskin:2011:statetodeltabx}.

\begin{definition}
Given small categories \(\A\) and \(\B\), a \textit{lens} \(F \colon \A \to \B\) consists of
\begin{itemize}
    \item a functor \(F \colon \A \to \B\), called the \textit{get functor} of \(F\), and
    \item a function \(\lift{F}{A} \colon \outSet{\B}{FA} \to \outSet{\A}{A}\) for each \(A \in \objectSet{\A}\), collectively known as the \textit{put functions},
\end{itemize}
such that
\begin{itemize}
    \item \textit{\PutGet{}}: \(F\lift{F}{A}b = b\) for all \(A \in \objectSet{\A}\) and all \(b \in \outSet{\B}{FA}\),
    \item \textit{\PutId{}}: \(\lift{F}{A}\id{FA} = \id{A}\) for all \(A \in \objectSet{\A}\), and
    \item \textit{\PutPut{}}: \(\lift{F}{A}(b' \compose b) = \lift{F}{A'}b' \compose \lift{F}{A}b\) for all \(A \in \objectSet{\A}\), \(b \in \outSet{\B}{FA}\), \(b' \in \outSet{\B}{FA'}\), where \(A' = \target \lift{F}{A}b\).
\end{itemize}
\end{definition}

\noindent There is a category \(\Lens\) whose objects are small categories and whose morphisms are lenses. The composite \(G \compose F\) of lenses \(F \colon \A \to \B\) and \(G \colon \B \to \C\) has get functor which is the composite of the get functors of \(G\) and \(F\), and has  \(\lift{(G \compose F)}{A}c = \lift{F}{A}\lift{G}{FA}c\) for all \(A \in \objectSet{\A}\) and all \(c \in \outSet{\C}{GFA}\). There is also an identity-on-objects forgetful functor \(\forget \colon \Lens \to \Cat\) that sends a lens to its get functor.

\begin{definition}
A functor \(F \colon \A \to \B\) is a \textit{discrete opfibration} if, for each \(A \in \objectSet{\A}\) and each \(b \in \outSet{\B}{FA}\), there is a unique \(a \in \outSet{\A}{A}\) such that \(Fa = b\).
\end{definition}

\begin{remark}
If \(F \colon \A \to \B\) is a discrete opfibration, then there is a unique lens mapped by \(\forget\) to \(F\). We will sometimes also use the name \(F\) to refer to this unique lens above \(F\).
\end{remark}

We also recall Johnson and Roseburgh's ``pullback'' of a cospan of lenses~\cite{JohnsonRosebrugh:2015:SpansDeltaLenses}, which we will refer to as their \textit{proxy pullback}, adopting the terminology of Bumpus and Kocsis~\cite{bumpus:2021:spined-categories:-generalizing-tree-width}.

\begin{definition}
The \textit{proxy pullback} of a lens cospan \(\A \xrightarrow{F} \C \xleftarrow{G}\B\) is a lens span \(\A \xleftarrow{\Gbar} \D \xrightarrow{\Fbar} \B\) where
\begin{itemize}
    \item the get functors of \(\Fbar\) and \(\Gbar\) form a pullback square
    \[\begin{tikzcd}
    \D \arrow[r, "\forget\Fbar"]\arrow[d, "\forget\Gbar" swap] & \B \arrow[d, "\forget G"]\\
    \A \arrow[r, "\forget F" swap] & \C
    \end{tikzcd}\]
    in \(\Cat\) (this determines them up to isomorphism), and
    \item for each \(D \in \objectSet{\D}\), each \(a \in \outSet[\big]{\A}{\Gbar D}\), and each \(b \in \outSet[\big]{\B}{\Fbar D}\),
    \[
    \lift{\Fbar}{D}b = \pair[\big]{\lift{F}{\Gbar D}Gb}{b}
    \qquad\text{and}\qquad
    \lift{\Gbar}{D}a = \pair[\big]{a}{\lift{G}{\Fbar D}Fa}.
    \]
\end{itemize}

\noindent When \(F = G\), the lenses \(\Fbar,\Gbar\colon \D \to \A\) are also called the \textit{proxy kernel pair} of \(F\).
\end{definition}

\section{Characterising monic and epic lenses}
\label{Section: Characterising monic and epic lenses}

\subsection*{Monic lenses}
\label{Section: Monic lenses}

We will study the monos in \(\Lens\) via their relation to those in \(\Cat\), expressed as follows.

\begin{theorem}
The functor \(\forget\) preserves and reflects monos.
\end{theorem}
Reflection was proved and preservation conjectured by Chollet et al.~\cite{Clarke:2021:CategoryLens}. Recalling that a morphism is monic if and only if it has a kernel pair with both morphisms equal, we may prove preservation.

\begin{proof}[Proof that \(\forget\) preserves monos.]
Let \(M \colon \A \to \B\) be a monic lens, and let \(P_1, P_2 \colon \Ker \forget M \to \A\) be its proxy kernel pair in \(\Lens\). As \(M\) is monic and \(M \compose P_1 = M \compose P_2\), actually \(P_1 = P_2\), and so \(\forget P_1 = \forget P_2\). But \(\forget P_1\) and \(\forget P_2\) are the (real) kernel pair of \(\forget M\) in \(\Cat\). Hence \(\forget M\) is a monic functor.
\end{proof}

Chollet et al.~\cite{Clarke:2021:CategoryLens} also showed that the get functor of a lens is monic if and only if it is a cosieve.

\begin{definition}
A \textit{cosieve} is an injective-on-objects discrete opfibration.
\end{definition}

\begin{corollary}
\label{Monic lenses are cosieves}
The functor \(\forget\) restricts to a bijection between monic lenses and cosieves.
\end{corollary}

\begin{proof}
A cosieve is a discrete opfibration, so there is a unique lens above it; by reflection, this lens is monic. Conversely, the get functor of a monic lens is, by preservation, monic, and so is a cosieve.
\end{proof}

The above result says that monic lenses and cosieves are essentially the same. We continue to use the term cosieve for functors when we wish to distinguish these from monic lenses.

\subsection*{Lens images and factorisation}
\label{Section: Lens images and factorisation}

The images of the object and morphism maps of a functor do not always form a subcategory of a functor's target category. The situation is nicer for the get functor of a lens \(F\); in this case, the images actually form an out-degree-zero subcategory \(\Image F\) of the lens' target category, which we will call the \textit{image} of \(F\). By \textit{out-degree-zero} subcategory, we mean one for which any morphism out of an object in the subcategory belongs to the subcategory. As cosieves are exactly the out-degree-zero subcategory inclusion functors, we obtain the following factorisation result.

\begin{proposition}
\label{Lens image factorisation}
    Every lens \(F \colon \A \to \B\) has a factorisation
    \[
    \begin{tikzcd}[column sep=large]
        \A \arrow[r, "E" swap] \arrow[rr, bend left=15, shift left=1, "F"] & \Image F \arrow[r, "M" swap] & \B
    \end{tikzcd}
    \]
    in \(\Lens\) where \(M\) is monic and \(E\) is surjective on objects and morphisms.
\end{proposition}

Recall that a morphism \(e \colon A \to B\) is \textit{left orthogonal} to a morphism \(m \colon C \to D\) if, for all pairs of morphisms \(f \colon A \to C\) and \(g \colon B \to D\) such that \(g \compose e = m \compose f\), there is a unique morphism \(h \colon B \to C\), called the \textit{diagonal filler}, such that \(f = h \compose e\) and \(g = m \compose h\). Also recall that classes \(\mathcal{E}\) and \(\mathcal{M}\) of morphisms form an \textit{orthogonal factorisation system} if \(\mathcal{E}\) is the class of all morphisms that are left orthogonal to all morphisms in \(\mathcal{M}\), and every morphism \(f\) factors as \(f = m \compose e\) for some \(e \in \mathcal{E}\) and some \(m \in \mathcal{M}\).

\begin{remark}
\label{Factorisation System}
The above factorisation is already known to Johnson and Roseburgh, who showed that the surjective-on-objects lenses and the injective-on-objects-and-morphisms lenses form an orthogonal factorisation system on \(\Lens\)~\cite{JohnsonRoseburgh:2021:TheMoreLegsTheMerrier}. Our addition is that this is actually an epi-mono factorisation system; we have already shown that the injective-on-objects-and-morphisms lenses are exactly the monic lenses, and we will show in the next section that the surjective-on-objects lenses are exactly the epic lenses. In \cref{Section: Regular epic lenses}, we will also deduce the orthogonality without explicitly constructing the diagonal fillers.
\end{remark}

\subsection*{Epic lenses}
\label{Section: Epic lenses}

We may also study the epis in \(\Lens\) via their relation to those in \(\Cat\).

\begin{theorem}
\label{U preserves epis}
The functor \(\forget\) preserves and reflects epis.
\end{theorem}

Again, reflection was proved and preservation conjectured by Chollet et al.~\cite{Clarke:2021:CategoryLens}. The first proof of preservation was sketched by Lack in an unpublished personal communication to Clarke; we present a new, simpler proof below. First, we recall some preliminary results about epic functors and epic lenses.

\begin{proposition}
\label{Epic functors}
Every epic functor is surjective on objects. Every functor that is surjective both on objects and on morphisms is epic.
\end{proposition}

Recall that not all epic functors are surjective on morphisms.

\begin{example}
\label{Example epic functor not surjective on morphisms}
Let \(J \colon \intervalCat \to \isomorphismCat\) be the functor that sends the non-identity morphism \(u\) of the interval category~\(\intervalCat\) to the morphism \(v\) of the free living isomorphism \(\isomorphismCat\). Then \(J\) is epic because any two functors out of \(\isomorphismCat\) which agree on \(v\) must also agree on \(v^{-1}\). However, the morphism \(v^{-1}\) is not in the image of \(J\).
\end{example}

\begin{proposition}
\label{Proxy cokernel pair in lens}
Let \(F \colon \A \to \B\) be a lens, and let \(\Jbar_1, \Jbar_2 \colon \B \to C\) be the cokernel pair of \(\forget F\). Then \(\Jbar_1\) and \(\Jbar_2\) are cosieves, and the unique lenses \(J_1\) and \(J_2\) above \(\Jbar_1\) and \(\Jbar_2\) satisfy \(J_1 \compose F = J_2 \compose F\).
\end{proposition}

\begin{proof}
Let \(F = M \compose E\) be the factorisation of \(F\) given in \cref{Lens image factorisation}. By \cref{Epic functors}, \(\forget E\) is an epic functor. As \(\Jbar_1 \compose \forget M \compose \forget E = \Jbar_1 \compose \forget F = \Jbar_2 \compose \forget F = \Jbar_2 \compose \forget M \compose \forget E\), actually \(\Jbar_1 \compose \forget M = \Jbar_2 \compose \forget M\). It follows that \(\Jbar_1\) and \(\Jbar_2\) are also the cokernel pair of \(\forget M\). As cosieves are pushout stable and \(\forget M\) is a cosieve, so are \(\Jbar_1\) and \(\Jbar_2\). As there is a unique lens above the discrete opfibration \(\Jbar_1 \compose \forget M = \Jbar_2 \compose \forget M\), we must have that \(J_1 \compose M = J_2 \compose M\).
\end{proof}

\begin{remark}
Later, we will see that \(J_1\) and \(J_2\) are actually a cokernel pair of \(F\) in \(\Lens\).
\end{remark}

\begin{proof}[Proof that \(\forget\) preserves epis.]
Let \(E \colon \A \to \B\) be an epic lens, and \(J_1\) and \(J_2\) the unique lenses above the cokernel pair of \(\forget E\) from \cref{Proxy cokernel pair in lens}. As \(J_1 \compose E =J_2 \compose E\) and \(E\) is epic, actually \(J_1 = J_2\), and so \(\forget J_1 = \forget J_2\). But \(\forget J_1\) and \(\forget J_2\) are the cokernel pair of \(\forget E\), so \(\forget E\) is also epic.
\end{proof}

\begin{corollary}
\label{Epic lenses are surjective on objects and morphisms}
Let \(F\) be a lens. Then the following are equivalent:
\begin{enumerate}[label=\normalfont(\arabic*)]
\item \(F\) is epic,
\item \(\forget F\) is surjective on objects,
\item \(\forget F\) is surjective on morphisms.
\end{enumerate}
\end{corollary}

\begin{proof}
Chollet et al.~\cite{Clarke:2021:CategoryLens} showed that (2) and (3) are equivalent, and imply (1). Suppose that \(F\) is epic. As \(\forget\) preserves epis (\cref{U preserves epis}), so is \(\forget F\). By \cref{Epic functors}, \(\forget F\) is surjective on objects.
\end{proof}

\section{Coequalisers of lenses}
\label{Section: Coequalisers of lenses}

Given morphisms \(f_1, f_2 \colon A \to B\),
we say that a morphism \(e \colon B \to C\) \textit{coforks} \(f_1\) and \(f_2\) if \(e \compose f_1 = e \compose f_2\). Some authors would use the verb coequalise where we use the verb cofork. Unlike those authors, we say that \(e\) \textit{coequalises} \(f_1\) and \(f_2\) only when \(e\) is universal among coforks of \(f_1\) and \(f_2\).

\subsection*{Non-existence, non-preservation and non-reflection of coequalisers}
\label{Section: Non-existence, non-preservation and non-reflection of coequalisers}

Recall that \(\Cat\) has all coequalisers. Shortly, we will construct several counterexamples to the well-behavedness of coequalisers in \(\Lens\), at least with respect to those in~\(\Cat\). To do this, we will use the following proposition, which gives necessary conditions for a cofork of lenses to be a coequaliser.

\begin{proposition}
\label{Necessary condition for existence of coequaliser}
Let \(F_1, F_2 \colon \A \to \B\) be lenses with coequaliser \(E \colon \B \to \C\) in \(\Lens\). Then
\begin{enumerate}[label=\normalfont(\arabic*)]
    \item 
    \label{Necessary condition 1}
    for each cofork \(G \colon \B \to \D\) of \(F_1\) and \(F_2\),
    \(\lift{G}{B}d = \lift{E}{B}E \lift{G}{B}d\)
    for all \(B \in \objectSet{\B}\) and all \(d \in \outSet{\D}{GB}\); and
    \item
    \label{Necessary condition 2}
    in particular, \(E\) is the unique lens above \(\forget E\) that coforks \(F_1\) and \(F_2\).
\end{enumerate}
\end{proposition}

\begin{proof}
For \ref{Necessary condition 1}, if \(G \colon \B \to \D\) coforks \(F_1\) and \(F_2\), then there is a lens \(H \colon \C \to \D\) such that \(G = H \compose E\), and so
\(\lift{G}{B}d = \lift{(H \compose E)}{B}d = \lift{E}{B}\lift{H}{E B}d = \lift{E}{B}E\lift{E}{B}\lift{H}{E B} d = \lift{E}{B}E \lift{(H \compose E)}{B} d = \lift{E}{B}E \lift{G}{B}d\). For \ref{Necessary condition 2}, if \(G \colon \B \to \C\) is a lens above \(\forget E\) that coforks \(F_1\) and \(F_2\), then 
\(\lift{G}{B}c = \lift{E}{B}E\lift{G}{B}c = \lift{E}{B}G\lift{G}{B}c = \lift{E}{B}c\)
for each \(B \in \objectSet{\B}\) and each \(c \in \outSet{\C}{EB}\), and so \(G = E\).
\end{proof}

The first example shows that \(\Lens\) does not have all coequalisers, nor does \(\forget\) reflect them.

\begin{example}
\label{Lens doesn't have all coequalisers}
Let \(\A\) and \(\B\) be the preordered sets generated respectively by the following graphs.
\begin{align*}
\begin{tikzcd}[ampersand replacement=\&]
Y_1 \& X \arrow[l, "f_1" swap]\arrow[r, "f_2"]\arrow[d, "f"] \& Y_2\\
\&Y\&
\end{tikzcd}
&&
\begin{tikzcd}[ampersand replacement=\&]
Y'_1 \& X' \arrow[l, "f'_1" swap]\arrow[r, "f'_2"]\arrow[d, phantom] \& Y'_2\\
\&\phantom{Y}\&
\end{tikzcd}
\end{align*}
Let \(F_1, F_2 \colon \A \to \B\) be the unique lenses that both send \(X\) to \(X'\), \(Y_1\) to \(Y'_1\), \(Y_2\) to \(Y'_2\), and such that \(F_1 Y = Y'_1\), \(\lift{F_1}{X}f'_1 = f_1\), \(F_2 Y = Y'_2\), and \(\lift{F_2}{X}{f'_2} = f_2\). Let \(G \colon \B \to \intervalCat\) be the unique functor that sends \(X'\) to \(0\), and both \(Y'_1\) and \(Y'_2\) to \(1\); \(G\) coequalises \(\forget F_1\) and \(\forget F_2\) in \(\Cat\). There are only two lens structures on \(G\) that cofork \(F_1\) and \(F_2\) in \(\Lens\); one is determined by \(\lift{G_1}{X'}u = f'_1\) and the other by \(\lift{G_2}{X'}u = f'_2\). By \cref{Necessary condition for existence of coequaliser}, neither \(G_1\) nor \(G_2\) coequalises \(F_1\) and \(F_2\). Thus \(\forget\) does not reflect the coequaliser \(G\) of \(\forget F_1\) and \(\forget F_2\).

Actually \(F_1\) and \(F_2\) do not have a coequaliser in \(\Lens\). Assume that \(E \colon \B \to \C\) is such a coequaliser. Then \(Ef'_1 = EF_1f = EF_2f = Ef'_2\). As \(G_1\) coforks \(F_1\) and \(F_2\), there is a lens \(H \colon \C \to \intervalCat\) such that \(G_1 = H \compose E\). As \(HEX' = G_1X' \neq G_1Y'_1 = HEY'_1\), we must have \(EX' \neq EY'_1\). Hence \(EX'\) and \(EY'_1\) are distinct objects of the image of \(E\), and \(\id{EX'}\), \(Ef'_1\) and \(\id{EY'_1}\) are distinct morphisms of the image of \(E\). As \(E\) is a coequaliser, it is epi, and so, by \cref{Epic lenses are surjective on objects and morphisms}, its image is all of \(\C\). Thus \(\forget H\) is an isomorphism in \(\Cat\), and so \(H\) is an isomorphism in \(\Lens\). Hence \(G_1\) also coequalises \(F_1\) and \(F_2\), which is a contradiction.
\end{example}

There are even parallel pairs of lenses for which the coequaliser of their get functors has a unique lens structure that coforks them, and yet does not coequalise them.

\begin{example}
Let \(\A\), \(\B\) and \(\C\) be the preordered sets generated respectively by the following graphs.
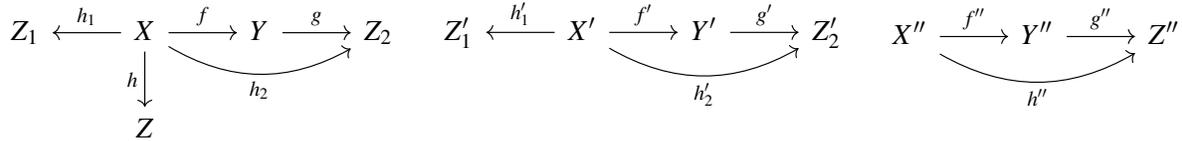
\begin{align*}
\begin{tikzcd}[ampersand replacement=\&]
Z_1\&
X\arrow[d, "h" swap]\arrow[rr, bend right, "h_2" swap]\arrow[r, "f"]\arrow[l, "h_1" swap]\&
Y \arrow[r, "g"]\&
Z_2
\\
\&Z
\end{tikzcd}
&&
\begin{tikzcd}[ampersand replacement=\&]
Z'_1\&
X'\arrow[d, phantom]\arrow[rr, bend right, "h'_2" swap]\arrow[r, "f'"]\arrow[l, "h'_1" swap]\&
Y' \arrow[r, "g'"]\&
Z'_2
\\
\&\phantom{Z}
\end{tikzcd}
&&
\begin{tikzcd}[ampersand replacement=\&]
X''\arrow[d, phantom]\arrow[rr, bend right, "h''" swap]\arrow[r, "f''"]\&
Y'' \arrow[r, "g''"]\&
Z''
\\
\phantom{Z}
\end{tikzcd}
\end{align*}
Let \(F_1, F_2 \colon \A \to \B\) be the unique lenses that both send \(X\) to \(X'\), \(Y\) to \(Y'\), \(Z_1\) to \(Z'_1\), \(Z_2\) to \(Z'_2\), and such that \(F_1Z = Z'_1\), \(\lift{F_1}{X}h_1' = h_1\) and \(F_2Z = Z'_2\). Let \(E \colon \B \to \C\) be the unique lens that sends \(X'\) to \(X''\), \(Y'\) to \(Y''\), and both \(Z'_1\) and \(Z'_2\) to \(Z''\). Then \(\forget E\) coequalises \(\forget F_1\) and \(\forget F_2\) in \(\Cat\), and \(E\) coforks \(F_1\) and \(F_2\) in \(\Lens\). However, \(E\) does not coequalise \(F_1\) and \(F_2\) in \(\Lens\). Indeed, if \(G \colon \B \to \intervalCat\) is the unique lens that sends \(X'\) to \(0\), all of \(Y'\), \(Z'_1\) and \(Z'_2\) to \(1\), and for which \(\lift{G}{X'}u = h'_1\), then \(\lift{E}{X'}E\lift{G}{X'}u = \lift{E}{X'}Eh'_1 = \lift{E}{X'}h'' = h'_2 \neq h'_1 = \lift{G}{X'}u\).
\end{example}

The final example shows that \(\forget\) does not preserve coequalisers. It also shows that there are parallel pairs of lenses for which the coequaliser of their get functors has no lens structure that coforks them.

\begin{example}
\label{U doesn't preserve all coequalisers}
Let \(\A\) be the preordered set generated by the graph
\[\begin{tikzcd}
Y_1 & X \arrow[l, "f_1" swap] \arrow[r, "f_2"] & Y_2
\end{tikzcd}\]
Let \(I \colon \A \to \A\) denote the identity lens, and let \(S \colon \A \to \A\) be the unique lens that maps \(X\) to \(X\), \(Y_1\) to \(Y_2\) and \(Y_2\) to \(Y_1\). The coequaliser of \(\forget I\) and \(\forget S\) in \(\Cat\) is the unique functor \(E \colon \A \to \intervalCat\) that sends \(X\) to \(0\) and both \(Y_1\) and \(Y_2\) to \(1\). Recall that \(\terminalCat\) is terminal in \(\Lens\)~\cite{Clarke:2021:CategoryLens}. We claim that the coequaliser of \(I\) and~\(S\) in \(\Lens\) is the unique lens \(E \colon \A \to \terminalCat\). Let \(G \colon \A \to \C\) be a lens that coforks \(I\) and \(S\) in \(\Lens\). Let \(f = G f_1\). Then \(f = G f_1 = GI f_1 = GS f_1 = G f_2\). As \(\lift{G}{X}{f} \in \outSet{\A}{X}\), it is one of \(f_1\), \(f_2\) and \(\id{X}\). If \(\lift{G}{X}{f} = f_1\), then
\[f_1 = \lift{I}{X}{f_1} = \lift{I}{X}\lift{G}{X}{f} = \lift{(G \compose I)}{X}{f} = \lift{(G \compose S)}{X}{f} = \lift{S}{X}\lift{G}{X}{f} = \lift{S}{X}{f_1} = f_2,\]
which is a contradiction. We get a similar contradiction if \(\lift{G}{X}{f} = f_2\). By elimination, \(\lift{G}{X}{f} = \id{X}\), and so \(f = G\lift{G}{X}{f} = G\id{X} = \id{G X}\). The image of \(G\) thus consists of the object \(G X\) and the morphism \(\id{G X}\). If \(H \colon \terminalCat \to \C\) is a lens such that \(G = H \compose E\), then \(H\) must send \(0\) to \(G X\), and this uniquely determines~\(H\). As the image of any lens, in particular \(G\), is an out-degree-zero subcategory of its target category, this definition of \(H\) does indeed give a lens, and \(G = H \compose E\). Of course, the factorisation \(G = H \compose E\) is really the image factorisation of \(G\) from \cref{Factorisation System}.
\end{example}

\subsection*{Coequalisers which are reflected}
\label{Section: Coequalisers which are reflected}

Although the counterexamples above suggest that coequalisers in \(\Lens\) have little relation to those in \(\Cat\), we will see in \cref{Pushouts in Lens} and \cref{Epic lens is regular} two classes of coequalisers in \(\Lens\) which do lie over coequalisers in \(\Cat\). The following theorem, a partial converse to \cref{Necessary condition for existence of coequaliser}, reduces checking the coequaliser property in these cases to checking that \cref{Equation: Condition for reflection} below always holds.

\begin{theorem}
\label{Condition for reflection}
Let \(F_1, F_2 \colon \A \to \B\) be lenses. Let \(E \colon \B \to \C\) be a cofork of \(F_1\) and \(F_2\) in \(\Lens\), and suppose that \(\forget E\) coequalises \(\forget F_1\) and \(\forget F_2\) in \(\Cat\). Then \(E\) coequalises \(F_1\) and \(F_2\) in \(\Lens\) if and only if for all lenses \(G \colon \B \to \D\) that cofork \(F_1\) and \(F_2\) in \(\Lens\), all \(B \in \objectSet{\B}\) and all \(d \in \outSet{\D}{GB}\), we have
\begin{equation}
    \label{Equation: Condition for reflection}
    \lift{G}{B}d = \lift{E}{B}E\lift{G}{B}d.
\end{equation}
\end{theorem}

In the proof of the following lemma and again, later, in the proof of \cref{Lemma: Pushout discrete opfibration and cosieve}, we use the induction principle for the equivalence relation \(\simeq\) on a set \(S\) generated by a binary relation \(R\) on \(S\), that is,
\begin{equation}
\label{Induction principle}
    \forall P \; x_0 \; y_0.\ 
    \left[
    \begin{aligned}
    &x_0 \simeq y_0\\
    &\quad \land \quad \forall x \; y. \:\: x \mathrel{R} y \implies P (x, y)\\
    &\quad \land \quad \forall x.\:\:  P(x, x)\\
    &\quad \land \quad \forall x \; y. \:\:  \brack{x \simeq y \;\land\; P(x, y)} \implies P(y, x)\\
    &\quad \land \quad \forall x \; y \; z. \:\:  \brack{x \simeq y \;\land\; P(x, y) \;\land\; y \simeq z \;\land\; P (y, z)} \implies P(x, z)
    \end{aligned}
    \right]
    \implies P(x_0, y_0).
\end{equation}

\begin{lemma}
\label{Comparison lens}
Let \(F_1, F_2 \colon \A \to \B\) be lenses. Let \(E \colon \B \to \C\) be a cofork of \(F_1\) and \(F_2\) in \(\Lens\), and suppose that \(\forget E\) coequalises \(\forget F_1\) and \(\forget F_2\) in \(\Cat\). Let \(G \colon \B \to \D\) be a lens that coforks \(F_1\) and \(F_2\) in \(\Lens\), and let \(H \colon \C \to \D\) be the unique functor such that \(\forget G = H \compose \forget E\). Then there is a unique lens structure on \(H\) that, for all \(B \in \objectSet{\B}\) and all \(d \in \outSet{\D}{GB}\), satisfies the equation
\begin{equation}
    \label{Equation: Definition of put}
    \lift{H}{EB}d = E\lift{G}{B}d.
\end{equation}
\end{lemma}

\begin{proof}
For each \(C \in \objectSet{\C}\), as \(\forget E\) is epic, there is a \(B \in \objectSet{\B}\) such that \(EB = C\). Hence, we may define \(\lift{H}{C}\) using \cref{Equation: Definition of put}, so long as, for all \(B_1, B_2 \in \objectSet{\B}\), if \(E B_1 = E B_2\) then, for all \(d \in \outSet{\D}{E B_1}\), we have \(E\lift{G}{B_1} d = E \lift{G}{B_2} d\). Let \(\simeq\) be the smallest equivalence relation on \(\objectSet{\B}\) such that \(F_1A \simeq F_2A\) for all \(A \in \objectSet{\A}\). As \(\forget E\) coequalises \(\forget F_1\) and \(\forget F_2\) in \(\Cat\), we have~\cite[Proposition~4.1]{Bednarczyk:1999:GeneralizedCongruences}, for all \(B_1, B_2 \in \objectSet{\B}\), that \(E B_1 = E B_2\) if and only if \(B_1 \simeq B_2\). We proceed using the induction principle in \cref{Induction principle}. The proof obligations from the reflexivity, symmetry and transitivity axioms for \(\simeq\) hold as~\(=\) is an equivalence relation. For the remaining one, for all \(A \in \objectSet{\A}\) and all \(d \in \outSet{\D}{F_1 A}\), we have
\[E\lift{G}{F_1 A} d
    = EF_1\lift{F_1}{A}\lift{G}{F_1 A} d
    = (E \compose F_1)\lift{(G \compose F_1)}{A} d
    = (E \compose F_2)\lift{(G \compose F_2)}{A} d
    = EF_2\lift{F_2}{A} \lift{G}{F_2 A}d
    = E\lift{G}{F_2 A} d.\]

Define \(\lift{H}{C}\) using \cref{Equation: Definition of put}. It remains to check that the lens laws hold for \(H\). For all \(C \in \objectSet{\C}\), there is a \(B \in \objectSet{\B}\) such that \(EB = C\), and \(\lift{H}{C}\id{HC} = E\lift{G}{B}\id{GB} = E\id{B} = \id{C}\); hence \PutId{} holds. For all \(C \in \objectSet{\C}\), all \(d \in \outSet{\D}{HC}\) and all \(d' \in \outSet{\D}{\target d}\), there is a \(B \in \objectSet{\B}\) such that \(EB = C\), and
\[\lift{H}{C}(d' \compose d)
= E\lift{G}{B}(d' \compose d)
= E\paren[\big]{\lift{G}{B'} d' \compose \lift{G}{B} d}
= E\lift{G}{B'}d' \compose E\lift{G}{B} d
= \lift{H}{C'} d' \compose \lift{H}{C} d,\]
where \(B' = \target \lift{G}{B} d\) and \(C' = EB'\); hence \PutPut{} holds. Finally, for all \(C \in \objectSet{\C}\) and all \(d \in \outSet{\D}{HC}\), there is a \(B \in \objectSet{\B}\) such that \(EB = C\), and \(H\lift{H}{C} d = HE\lift{G}{B} d = G\lift{G}{B} d = d\); hence \PutGet{} holds.
\end{proof}

\begin{proof}[Proof of \cref{Condition for reflection}.]
We proved the \textit{only if} direction in \cref{Necessary condition for existence of coequaliser}. For the \textit{if} direction, suppose, for all lenses \(G \colon \B \to \D\) that cofork \(F_1\) and \(F_2\), that \cref{Equation: Condition for reflection} always holds. We must show that \(E\) is the universal cofork of \(F_1\) and \(F_2\) in \(\Lens\). Let \(G \colon \B \to \D\) be another cofork of \(F_1\) and \(F_2\) in \(\Lens\). Suppose that there is a lens \(H \colon \C \to \D\) such that \(G = H \compose E\). Then \(\forget G = \forget H \compose \forget E\), and so \(\forget H\) is the unique functor that composes with \(\forget E\) to give \(\forget G\). Let \(C \in \objectSet{\C}\) and \(d \in \outSet{\D}{HC}\). As \(\forget E\) is epic, there is a \(B \in \objectSet{\B}\) such that \(EB = C\). Then \(\lift{H}{C}d = E\lift{E}{B}\lift{H}{C}d = E \lift{(H \compose E)}{B}d = E \lift{G}{B}d\). Hence \(H\) is uniquely determined. Now let \(H \colon \C \to \D\) be the lens defined as in \cref{Comparison lens}. For all \(B \in \objectSet{\B}\) and all \(d \in \outSet{\D}{GB}\), we have \(\lift{G}{B}d = \lift{E}{B}E\lift{G}{B}d = \lift{E}{B}\lift{H}{EB}d = \lift{(H \compose E)}{B} d\), and so \(G = H \compose E\).
\end{proof}

\begin{corollary}
\label{Sufficient condition for reflection}
Let \(F_1, F_2 \colon \A \to \B\) be lenses. Let \(E \colon \B \to \C\) be a cofork of \(F_1\) and \(F_2\) in \(\Lens\), and suppose that \(\forget E\) coequalises \(\forget F_1\) and \(\forget F_2\). If \(\forget E\) is a discrete opfibration then \(E\) coequalises \(F_1\) and \(F_2\).
\end{corollary}

\begin{proof}
Let \(G \colon \B \to \D\) be a lens that coforks \(F_1\) and \(F_2\), let \(B \in \objectSet{\B}\) and let \(d \in \outSet{\D}{G B}\). Then \(\lift{G}{B}d\) and \(\lift{E}{B}E\lift{G}{B}d\) are both elements of \(\outSet{\B}{B}\) which are sent by \(E\) to the same morphism \(E\lift{G}{B}d\) of \(\C\). If \(\forget E\) is a discrete opfibration, then \(E\lift{G}{B}d\) has a unique lift to \(\outSet{\B}{B}\), and so  \(\lift{G}{B}d\) and \(\lift{E}{B}E\lift{G}{B}d\) must be equal.
\end{proof}

\section{Pushouts of discrete opfibrations along monos}
\label{Section: Pushouts of discrete opfibrations along monos}

In the proof that \(\forget\) preserves epis (\cref{U preserves epis}), we used the well-known result that cosieves are pushout stable to explain why the pushout in \(\Cat\) of the get functors of a span of monic lenses lifts uniquely to a commutative square in \(\Lens\); this lifted square is actually a pushout square in \(\Lens\). In this section, we will show, more generally, that \(\Lens\) has pushouts of discrete opfibrations along monics, and that \(\forget\) creates these pushouts. In what follows, we use square brackets for equivalence classes of elements.

Fritsch and Latch~\cite[Proposition~5.2]{FritschLatch:1981:HomotopyInversesForNerve} explicitly construct the pushout in \(\Cat\) of a functor along a full monic functor. Specialising to when the full monic functor is a cosieve, and recalling that the image of a cosieve is out-degree-zero, we obtain the following simplification of Fritsch~and~Latch's construction. 

\begin{proposition}
\label{Pushout of cosieve in Cat}
Let \(F \colon \A \to \C\) be a functor and \(J \colon \A \to \B\) be a cosieve. Then
\[\begin{tikzcd}
\A \arrow[r, tail, "J"] \arrow[d, "F" swap] & \B \arrow[d, "\Fbar"]\\
\C \arrow[r, tail, "\Jbar" swap] & \D
\end{tikzcd}\]
is a pushout square in \(\Cat\) and \(\Jbar\) is a cosieve, where \(\D\), \(\Fbar\) and \(\Jbar\) are defined as follows:
\begin{itemize}
    \item \emph{Object set:}
    \[\objectSet{\D} = \objectSet{\C} \sqcup \paren[\big]{\objectSet{\B} \setDifference \objectSet{\A}}\]
    \item \emph{Hom-sets:} for all \(C_1, C_2 \in \objectSet{\C}\) and all \(B_1, B_2 \in \objectSet{\B} \setDifference \objectSet{\A}\),
    \begin{align*}
        \homSet{\D}{C_1}{C_2} &= \homSet{\C}{C_1}{C_2}&
        \homSet{\D}{C_1}{B_2} &= \emptyset\\
        \homSet{\D}{B_1}{B_2} &= \homSet{\B}{B_1}{B_2}&
        \homSet{\D}{B_1}{C_2} &= \paren[\big]{\Coprod_{A \in \objectSet{\A}} \homSet{\C}{FA}{C_2} \times \homSet{\B}{B_1}{A}} \big\slash {\sim}
    \end{align*}
    where \(\sim\) is the equivalence relation on \(\smallcoprod_{A \in \objectSet{\A}} \homSet{\C}{FA}{C_2} \times \homSet{\B}{B_1}{A}\) generated by \((c, a\compose b) \sim (c \compose F a, b)\) for all \(A_1, A_2 \in \objectSet{\A}\), all \(b \in \homSet{\B}{B_1}{A_1}\), all \(a \in \homSet{\A}{A_1}{A_2}\) and all \(c \in \homSet{\C}{FA_2}{C_2}\).
    \item \emph{Composition:} for all \(B_1, B_2, B_3 \in \objectSet{\B} \setDifference \objectSet{\A}\), all \(A \in \objectSet{\A}\), all \(C_1, C_2, C_3 \in \objectSet{\C}\), all \(b_1 \in \homSet{\D}{B_1}{B_2}\), all \(b_2 \in \homSet{\D}{B_2}{B_3}\), all \(a \in \homSet{\D}{B_2}{A}\), all \(c \in \homSet{\D}{FA}{C_2}\), all \(c_1 \in \homSet{\D}{C_1}{C_2}\) and all \(c_2 \in \homSet{\D}{C_2}{C_3}\),
    \begin{align*}
        b_2 \compose_\D b_1 &= b_2 \compose_\B b_1&
        \brack{(c, a)} \compose_\D b_1 &= \brack{(c, a \compose_\B b_1)}\\
        c_2 \compose_\D c_1 &= c_2 \compose_\C c_1 &
        c_2 \compose_\D \brack{(c, a)} &= \brack{(c_2 \compose_\C c, a)}
    \end{align*}
    \item \emph{Identity morphisms:} same as in \(\B\) and \(\C\).
    \item \emph{Injections:} the functor \(\Jbar \colon \C \to \D\) is the obvious inclusion of \(\C\) as a full subcategory of \(\D\); the functor \(\Fbar \colon \B \to \D\) is defined, for all \(B, B' \in \objectSet{\B} \setDifference \objectSet{\A}\), all \(A, A' \in \objectSet{\A}\), all \(b \in \homSet{\B}{B}{B'}\), all \(b' \in \homSet{\B}{B}{A}\) and all \(a \in \homSet{\B}{A}{A'}\), as follows:
    \begin{align*}
        \Fbar B &= B &&& \Fbar A &= F A \\
        \Fbar b &= b & \Fbar b' &= \brack{(\id{FA}, b')} & \Fbar a &= F a
    \end{align*}
\end{itemize}
\end{proposition}

\begin{theorem}
\label{Pushout discrete opfibration and cosieve}
The pushout in \(\Cat\) of a discrete opfibration along a cosieve is a discrete opfibration.
\end{theorem}

\begin{lemma}
\label{Lemma: Pushout discrete opfibration and cosieve}
Let \(F \colon \A \to \C\) be a discrete opfibration, let \(J \colon \A \to \B\) be a cosieve, let \(B \in \objectSet{\B} \setDifference \objectSet{\A}\) and let \(C \in \objectSet{\C}\). Then, for all \(A_1, A_2 \in \A\), all \(b_1 \in \homSet{\B}{B}{A_1}\), all \(b_2 \in \homSet{\B}{B}{A_2}\), all \(c_1 \in \homSet{\C}{FA_1}{C}\) and all \(c_2 \in \homSet{\C}{FA_2}{C}\), if \((c_1, b_1) \sim (c_2, b_2)\) then \(\lift{F}{A_1}c_1 \compose b_1 = \lift{F}{A_2}c_2 \compose b_2\).
\end{lemma}

\begin{proof}
We proceed by induction, using the induction principle for \(\sim\) in \cref{Induction principle}. The proof obligations from the reflexivity, symmetry and transitivity axioms for \(\sim\) hold because \(=\) is an equivalence relation. For the remaining proof obligation, for all \(A_1, A_2 \in \objectSet{\A}\), all \(b \in \homSet{\B}{B}{A_1}\), all \(a \in \homSet{\A}{A_1}{A_2}\) and all \(c \in \homSet{\C}{FA_2}{C}\), we have \(\lift{F}{A_1}Fa = a\) as \(F\) is a discrete opfibration, and so
\[\lift{F}{A_2}c \compose (a \compose b) = \lift{F}{A_2}c \compose \lift{F}{A_1}Fa \compose b = \lift{F}{A_1} (c \compose Fa) \compose b.\qedhere\]
\end{proof}

\begin{proof}[Proof of \cref{Pushout discrete opfibration and cosieve}.]
Using the notation of \cref{Pushout of cosieve in Cat}, suppose that \(F\) is a discrete opfibration. We must show that \(\Fbar\) is also a discrete opfibration. Let \(B \in \objectSet{\B}\) and \(d \in \outSet{\D}{\Fbar B}\).

Suppose that \(B \in \objectSet{\A}\). Then \(\Fbar B = F B\), and \(d \in \outSet{\C}{FB}\). As \(F\) is a discrete opfibration, there is a unique \(a \in \outSet{\A}{B}\) such that \(d = Fa\). But \(\outSet{\A}{B} = \outSet{\B}{B}\) as \(\A\) is out-degree-zero in \(\B\); also \(\Fbar a = F a\) for each \(a \in \outSet{\B}{B}\). Hence there is a unique \(a \in \outSet{\B}{B}\) such that \(d = \Fbar a\).

Suppose that \(B \in \objectSet{\B} \setDifference \objectSet{\A}\) and \(\target d \in \objectSet{\B} \setDifference \objectSet{\A}\). Then \(\Fbar B = B\), \(d \in \outSet{\B}{B}\) and \(\Fbar d = d\). As \(\Fbar\) is injective on the morphisms of \(\B\) not in \(\A\), \(d\) is the unique morphism in \(\outSet{\B}{B}\) mapped by \(\Fbar\) to~\(d\).
        
Otherwise, \(B \in \objectSet{\B} \setDifference \objectSet{\A}\) and \(\target d \in \objectSet{\C}\). Then \(\Fbar B = B\), and \(d = \brack{(c_1, b_1)}\) for some \(A_1 \in \objectSet{\A}\), some \(b_1 \in \homSet{\B}{B}{A_1}\) and some \(c_1 \in \homSet{\C}{FA_1}{C}\), where \(C = \target d\). For uniqueness of lifts, suppose that \(b_2 \in \outSet{\B}{B}\) is such that \(d = \Fbar b_2\). Let \(A_2 = \target b_2\). Then \(A_2 \in \objectSet{\A}\) as \(\Fbar A_2 = \target d = C\), and so \(\Fbar b_2 = \brack{(\id{C},b_2)}\). As \(d = \Fbar b_2\), we have \((\id{C},b_2) \sim (c_1, b_1)\). By \cref{Lemma: Pushout discrete opfibration and cosieve}, \(b_2 = \lift{F}{A_2}\id{C} \compose b_2 = \lift{F}{A_1} c_1 \compose b_1\); this determines~\(b_2\). For existence of lifts, note that \(\Fbar (\lift{F}{A_1} c_1 \compose b_1) = \brack{(\id{C},\lift{F}{A_1} c_1 \compose b_1)} = \brack{(F\lift{F}{A_1} c_1, b_1)} = \brack{(c_1, b_1)} = d\).
\end{proof}

\begin{theorem}
\label{Pushouts in Lens}
The functor \(\forget\) creates pushouts of monic lenses with discrete opfibrations.
\end{theorem}

\begin{proof}
Using the notation of \cref{Pushout of cosieve in Cat}, suppose that \(F\) is a discrete opfibration. Then \(\Fbar\) is also a discrete opfibration (\cref{Pushout discrete opfibration and cosieve}). Let \(J_\B \colon \B \to \B \sqcup \C\) and \(J_\C \colon \C \to \B \sqcup \C\) be the coproduct injection functors. Coproduct injections in \(\Cat\) are always discrete opfibrations, as is the coproduct copairing of any two discrete opfibrations. Hence \(J_\B\), \(J_\C\) and \(\copair{\Jbar}{\Fbar}\) are all discrete opfibrations. As the composite of two discrete opfibrations is a discrete opfibration, so are \(J_\B \compose J\) and \(J_\C \compose F\). So far, we know that \(\copair{\Jbar}{\Fbar}\) is the coequaliser in \(\Cat\) of \(J_\B \compose J\) and \(J_\C \compose F\), all of these functors have canonical lens structures as they are discrete opfibrations, and \(\copair{\Jbar}{\Fbar}\) coforks \(J_\B \compose J\) and \(J_\C \compose F\) in \(\Lens\). As \(\copair{\Jbar}{\Fbar}\) is a discrete opfibration, the conditions of \cref{Condition for reflection} are satisfied, and so \(\copair{\Jbar}{\Fbar}\) coequalises \(J_\B \compose J\) and \(J_\C \compose F\) in \(\Lens\). As \(\forget\) creates coproducts~\cite{Clarke:2021:CategoryLens}, it follows that \(\Jbar\) and \(\Fbar\) exhibit \(\D\) as the pushout of \(J\) and \(F\) in \(\Lens\).
\end{proof}

One might hope that the above result generalises to pushouts of two discrete opfibrations, or of arbitrary lenses along monics; this is not the case. The following is an example of two discrete opfibrations whose pushout injection functors have no lens structures that give a commutative square of lenses.

\begin{example}
Let \(\A\) and \(\B\) be the preordered sets generated respectively by the following graphs.
\begin{align*}
\begin{tikzcd}[ampersand replacement=\&, row sep={0em}]
Y'_1 \& X' \arrow[l, "f_1'"{swap}] \arrow[r, "f_2'"] \& Y_2' \\ Y_1'' \& X'' \arrow[l, "f_1''"] \arrow[r, "f_2''" swap] \& Y_2''
\end{tikzcd}
&&
\begin{tikzcd}[ampersand replacement=\&]
Y_1 \& X \arrow[l, "f_1"{swap}] \arrow[r, "f_2"] \& Y_2
\end{tikzcd}
\end{align*}
Let \(F \colon \A \to \B\) be the unique functor that sends both \(X'\) and \(X''\) to \(X\), both \(Y_1'\) and \(Y_1''\) to \(Y_1\), and both \(Y_2'\) and \(Y_2''\) to \(Y_2\). Let \(G \colon \A \to \B\) be the unique functor that sends both \(X'\) and \(X''\) to \(X\), both \(Y_1'\) and \(Y_2''\) to \(Y_1\), and both \(Y_2'\) and \(Y_1''\) to \(Y_2\). Both \(F\) and \(G\) are discrete opfibrations. Their pushout in \(\Cat\) is \(\intervalCat\); the pushout injections \(\Fbar, \Gbar \colon \B \to \intervalCat\) are both the unique functor that sends \(X\) to \(0\), and both \(Y_1\) and \(Y_2\) to~\(1\). There are two different lens structures on this functor; one lifts the unique morphism \(u\) of \(\intervalCat\) to \(f_1\), the other lifts it to~\(f_2\). This gives four different combinations of lens structures on \(\Fbar\) and \(\Gbar\). Assume, for a contradiction, that one of these combinations satisfies \(\Fbar G = \Gbar F\) in \(\Lens\). As \(\lift{G}{X'}\lift{\Fbar}{X}u = \lift{F}{X'}\lift{\Gbar}{X}u\), we must have \(\lift{\Fbar}{X}u = \lift{\Gbar}{X}u\). If \(\lift{\Fbar}{X}u = f_1\), then \(\lift{G}{X''}\lift{\Fbar}{X}u = \lift{G}{X''}f_1 = f_2'\) and \(\lift{F}{X''}\lift{\Gbar}{X}u = \lift{F}{X''}f_1 = f_1' \neq f_2'\), which is a contradiction. If \(\lift{\Fbar}{X}u = f_2\), we obtain a similar contradiction.
\end{example}

Next is an example of a lens and a cosieve where the pushout of the get functor of the lens along the cosieve does not have a lens structure (incidentally this lens and cosieve do not have a pushout in \(\Lens\)).

\begin{example}
Let \(\B\) and \(\D\) be the preordered sets generated respectively by the following graphs.
\begin{align*}
\begin{tikzcd}[ampersand replacement=\&]
X \arrow[d, "s" swap] \& W \arrow[l, "f" swap] \arrow[r, "g"] \arrow[d] \& Y \arrow[d, "t"]\\
Z_2 \& Z_1 \arrow[l] \arrow[r]\& Z_3
\end{tikzcd}
&&
\begin{tikzcd}[ampersand replacement=\&]
X' \arrow[dr, "s'" swap] \& W' \arrow[l, "f'" swap] \arrow[r, "g'"] \arrow[d] \& Y' \arrow[dl, "t'"]\\
\& Z'\&
\end{tikzcd}
\end{align*}
Let \(\A\) be the out-degree-zero subcategory of \(\B\) on the objects \(Z_1\), \(Z_2\) and \(Z_3\), and let \(J \colon \A \monicTo \B\) be the inclusion lens. As \(\terminalCat\) is terminal in \(\Lens\)~\cite{Clarke:2021:CategoryLens}, there is a unique lens \(F \colon \A \to \terminalCat\). By \cref{Pushout of cosieve in Cat}, the pushout of \(\forget F\) along \(\forget J\) in \(\Cat\) is the unique functor \(\Fbar \colon \B \to \D\) that maps \(W\) to \(W'\), \(X\) to \(X'\), \(Y\) to \(Y'\), and all of \(Z_1\), \(Z_2\) and \(Z_3\) to~\(Z'\). The functor \(\Fbar\) has no lens structure, otherwise we could derive the contradiction
\[s \compose f = \lift{\Fbar}{X}s' \compose \lift{\Fbar}{W} f' = \lift{\Fbar}{W} (s' \compose f') = \lift{\Fbar}{W} (t' \compose g') = \lift{\Fbar}{Y}t' \compose \lift{\Fbar}{W} g' = t \compose g.\]
\end{example}

From \cref{Pushouts in Lens}, every monic lens has a cokernel pair. Actually, using the epi-mono factorisation, every lens has a cokernel pair, namely, the cokernel pair of its mono factor.

\begin{proposition}
Every monic lens is effective (i.e.\ equalises its cokernel pair).
\end{proposition}

\begin{proof}
Let \(M \colon \A \to \B\) be a monic lens, and let \(J_1, J_2 \colon \B \to \Coker M\) be its cokernel pair. Based on \cref{Pushout of cosieve in Cat}, if \(B \in \objectSet{\B}\) is such that \(J_1B = J_2B\), then \(B \in \objectSet{\A}\); and similarly for morphisms of \(\B\). In particular, the image of any lens which forks \(J_1\) and \(J_2\) is contained in \(\A\), and thus its corestriction to \(\A\) is the unique comparison lens.
\end{proof}

\begin{corollary}
In \(\Lens\), the classes of all monos, effective monos, regular monos, strong monos and extremal monos coincide.
\end{corollary}

\begin{corollary}
Every lens that is both epic and monic is an isomorphism.
\end{corollary}

\section{Regular epic lenses}
\label{Section: Regular epic lenses}

In this section, we show that all epis in \(\Lens\) are regular. This gives us another class of coequalisers in \(\Lens\), namely, the epic lenses. For contrast, recall that not all epis in \(\Cat\) are regular.

\begin{example}
In \cref{Example epic functor not surjective on morphisms}, we saw that the functor \(J \colon \intervalCat \to \isomorphismCat\) is epic. It is, however, not a regular epi. Indeed, if \(J\) coforks \(F_1, F_2 \colon \A \to \intervalCat\), then \(F_1 = F_2\) as \(J\) is monic, and so \(\id{\intervalCat}\) is the coequaliser of \(F_1\) and \(F_2\), but \(\intervalCat\) and \(\isomorphismCat\) are not isomorphic.
\end{example}

\begin{proposition}
    \label{Epic lens over effective epic functor}
    The get functor of every epic lens is an effective epi in~\(\Cat\).
\end{proposition}

A functor \(E \colon \B \to \C\) is \textit{surjective on composable pairs} if for each composable pair \((c, c')\) of \(\C\), there is a composable pair \((b, b')\) of \(\B\) such that \(Eb = c\) and \(Eb' = c'\); such functors are necessarily also surjective on objects and morphisms. If \(E \colon \B \to \C\) is an epic lens, then \(\forget E\) is surjective on composable pairs; indeed, if \((c, c')\) is a composable pair of \(\C\), then there is a \(B \in \objectSet{\B}\) such that \(EB = \source c\), and \((\lift{E}{B}c, \lift{E}{\target \lift{E}{B}c}c')\) is a composable pair above \((c, c')\). Hence it suffices to prove the following lemma.

\begin{lemma}
\label{Effective epic functors}
All functors that are surjective on composable pairs are effective epis in \(\Cat\).
\end{lemma}

\begin{proof}
Let \(E \colon \B \to \C\) be a functor that is surjective on composable pairs, and let its kernel pair be \(F_1, F_2 \colon \Ker E \to \B\). We must show that \(E\) coequalises \(F_1\) and \(F_2\). Let \(G \colon \B \to \D\) cofork \(F_1\) and~\(F_2\).

Suppose that there is a functor \(H \colon \C \to \D\) such that \(G = H \compose E\). As \(E\) is surjective on objects, for all \(C \in \objectSet{\C}\) there is a \(B \in \objectSet{\B}\) such that \(EB = C\), and so \(HC = HEB = GB\); this equation determines \(H\) on objects. As \(E\) is surjective on morphisms, a similar equation determines \(H\) on morphisms.

To define \(H \colon \C \to \D\) with these equations, the values of \(GB\) and \(Gb\) should be independent of the choice of \(B\) above \(C\) and \(b\) above \(c\). For all \(C \in \objectSet{\C}\) and all \(B, B' \in \objectSet{\B}\) such that \(EB = EB' = C\), we have \(G B = GF_1 \pair{B}{B'} = GF_2 \pair{B}{B'} = G B'\), where \(\pair{B}{B'} \in \objectSet{\Ker E}\) comes from the pullback property; hence the object map of \(H\) is well defined. Its morphism map is similarly also well defined.

Define \(H\) with the above equations. By construction, \(G = H \compose E\). We must show that \(H\) is a functor. For all \(C \in \objectSet{\C}\), there is a \(B \in \objectSet{\B}\) such that \(EB = C\), and
\(H \id{C} = G \id{B} = \id{G B} = \id{H C}\); thus \(H\) preserves identities. For all composable pairs \(c\) and \(c'\) of \(\C\), there is a composable pair \(b\) and \(b'\) of \(\B\) such that \(Eb = c\) and \(Eb' = c'\), and \(H(c' \compose c) = G(b' \compose b) = Gb' \compose G b = Hc' \compose H c\); thus \(H\) preserves composites.
\end{proof}

\begin{corollary}
\label{Epic lens is regular}
Every epic lens coequalises its proxy kernel pair, and so is regular.
\end{corollary}

\begin{proof}
Let \(E \colon \B \to \C\) be an epic lens. Let \(F_1, F_2 \colon \Ker \forget E \to \B\) be the proxy kernel pair of \(E\) in \(\Lens\). By \cref{Epic lens over effective epic functor}, \(\forget E\) coequalises \(\forget F_1\) and \(\forget F_2\) in \(\Cat\). Let \(G \colon \B \to \D\) be a lens that coforks \(F_1\) and \(F_2\), let \(B \in \objectSet{\B}\), let \(d \in \outSet{\D}{GB}\), and let \(C = EB\). Then \(\lift{(G \compose F_1)}{\pair{B}{B}} d
= \lift{F_1}{\pair{B}{B}}\lift{G}{B} d
= \pair[\big]{\lift{G}{B} d}{\lift{E}{B}E\lift{G}{B} d}\), and similarly \(\lift{(G \compose F_2)}{\pair{B}{B}}{d} = \pair[\big]{\lift{E}{B}E\lift{G}{B} d}{\lift{G}{B}d}\). As \(G\) coforks \(F_1\) and \(F_2\), we have \(\lift{G}{B} d = \lift{E}{B}E\lift{G}{B} d\). By \cref{Condition for reflection}, \(E\) coequalises \(F_1\) and \(F_2\) in \(\Lens\).
\end{proof}

\begin{corollary}
In \(\Lens\), the classes of all epis, regular epis, strong epis and extremal epis coincide.
\end{corollary}

\begin{corollary}
In \(\Lens\), the class of all morphisms that are left orthogonal to the class of all monos is the class of all epis.
\end{corollary}

\begin{proof}
As \(\Lens\) has equalisers~\cite{Clarke:2021:CategoryLens}, every morphism that is left orthogonal to the class of all monos is an epi. Conversely, we have already shown that every epi is a strong epi.
\end{proof}

\begin{remark}
As every lens factors as an epi followed by a mono~(\cref{Factorisation System}), it follows that the class of all epis and the class of all monos together form an orthogonal factorisation system on \(\Lens\).
\end{remark}

\section{Conclusion}
\label{Section: Conclusion}

In this article, we have seen a number of results which advance our understanding of the category \(\Lens\) of (asymmetric delta) lenses. We now have a complete elementary characterisation of the monos and epis in \(\Lens\), the monos being the unique lenses on cosieves and the epis being the surjective-on-objects lenses; from this, we see that Johnson and Roseburgh’s factorisation system on \(\Lens\)~\cite{JohnsonRoseburgh:2021:TheMoreLegsTheMerrier} is actually an epi-mono factorisation system. We have also initiated a study of the coequalisers in \(\Lens\). Despite \(\Lens\) not having all coequalisers, nor the forgetful functor from \(\Lens\) to \(\Cat\) preserving or reflecting them, we have two interesting positive results. First, every epic lens coequalises its proxy kernel pair. Second, \(\Lens\) has pushouts of discrete opfibrations along cosieves. Our characterisation of the epic lenses played a central role in the proof of both of these results, and hopefully will enable future work to completely characterise the coequalisers in \(\Lens\).

That every epic lens coequalises its proxy kernel pair is yet another result that emphasises the parallels between proxy pullbacks in \(\Lens\) and real pullbacks in other categories. An interesting question for future work is whether there is an axiomatisation of the notion of proxy pullback from which one may prove interesting general results which also apply to other categories. Existing work in this direction include Bumpus and Kocsis’ \textit{proxy pushout}~\cite{bumpus:2021:spined-categories:-generalizing-tree-width}, which inspired our use of the name proxy pullback, as well as Böhm’s \textit{relative pullbacks}~\cite{bohm:2019:crossed-modules-monoids-relative} and Simpson’s \textit{local independent products}~\cite{simpson:2018:category-theoretic-structure-for-independence}. One potential use for such an axiomatised proxy pullback would be to give a generalised notion of regular category; the category \(\Lens\) is an obvious candidate example from which to draw inspiration. This notion of a proxy regular category may even be helpful for understanding symmetric lenses, which are known to be equivalence classes of spans of asymmetric ones, as some kind of relations in \(\Lens\), although this is as yet merely speculation.

\section*{Acknowledgements}
Many thanks go to my supervisors Michael Johnson, Richard Garner and Samuel Muller for their continual support; especially to Michael for his guidance, encouragement and suggestions, and all of the time he has dedicated to our regular meetings. I would also like to thank Bryce Clarke; his presentation to the Australian Category Seminar about the progress made at the Applied Category Theory Adjoint School in 2020 was what led me to think about monic and epic lenses, and the rest of the ideas in this paper followed from there; his feedback on my work has also been very helpful.

\newpage

\raggedright
\bibliographystyle{eptcs}
\bibliography{references}
\end{document}